\documentclass[ECP, preprint]{ejpecp} 

\usepackage[english]{babel}
\usepackage[bf]{caption}

\SHORTTITLE{Infinitely many collisions between a recurrent SRW and arbitrary many transient RWRE}

\TITLE{Infinitely many collisions between a recurrent simple random walk
and arbitrary many transient random walks in
a subballistic
random environment
    } 

\AUTHORS{%
  Alexis~Devulder\footnote{Universit\'e Paris-Saclay, UVSQ, CNRS, Laboratoire de Math\'ematiques de Versailles, 78000 Versailles, France.
    \EMAIL{alexis.devulder@universite-paris-saclay.fr}}
  }

\KEYWORDS{random walk; random environment; collisions; recurrence; transience} 

\AMSSUBJ{60K37; 60G50; 60J10} 

\SUBMITTED{December 8, 2024} 
\ACCEPTED{April 14, 2025} 

\VOLUME{0}
\YEAR{0}
\PAPERNUM{0}
\DOI{10.1214/YY-TN}

\ABSTRACT{We consider $d$ random walks
$\big(S_n^{(j)}\big)_{n\in\N}$, $1\leq j \leq d$,
in the same random environment $\omega$ in $\Z$, and a recurrent simple random walk $(Z_n)_{n\in\N}$ on $\Z$.
We assume that, conditionally on the environment $\omega$, all the random walks are independent and start from even initial locations.
Our assumption on the law of the environment is such that a single random walk in the environment $\omega$
is transient to the right but subballistic,
with parameter $0<\kappa<1/2$.
We show that - for every value of $d$ -
there are almost surely infinitely many times for which all these random walks,
$(Z_n)_{n\in\N}$
and $\big(S_n^{(j)}\big)_{n\in\N}$, $1\leq j \leq d$,
are simultaneously at the same location, even though one of them is recurrent and the $d$ others ones are transient.
}

\newcommand{\e} {\varepsilon}
\newcommand{\p} {\textnormal{\textsf{P}}}
\newcommand{\E} { \textnormal{\textsf{E}}}

\newcommand{\un}{\mathds{1}}

\newcommand{\N} { \mathbb{N} }

\newcommand{\Z} { \mathbb{Z} }
\newcommand{\R} { \mathbb{R} }

\renewcommand{\P} { \mathbb{P} }
\renewcommand{\k}{\kappa}
\renewcommand{\o}{\omega}
\newcommand{\po}{P_\omega}
\newcommand{\eo}{E_\o}

\begin{document}



\section{Introduction and statement of the main results}

\subsection{About collisions of random walks}
We consider several random walks.
We wonder if there exist almost surely infinitely many times
at which all these random walks are simultaneously at the same location;
such meetings are
called collisions.

This question was first considered for two (symmetric) simple random walks on $\Z^d$  by
P{\'o}lya \cite{Polya}. In fact (see P{\'o}lya \cite{PolyaInc}, "two incidents"),
the main motivation for  P{\'o}lya \cite{Polya} to prove his celebrated result about recurrence and transience
of simple random walks in $\Z^d$ was  to deduce that, almost surely,
two independent simple random walks in $\Z^d$ meet (or collide) infinitely often if $d\leq 2$,
while they meet only a finite number of times if $d\geq 3$.

Later, Dvoretzky and Erd{\"o}s \cite{Dvoretzky_Erdos} proved that almost surely,
three independent simple random walks in $\Z$
collide infinitely often, while four of them  collide only a finite number of times.

The question whether two or three random walks collide infinitely often has attracted
renewed attention recently.
See for example Krishnapur and Peres \cite{KrishnapurPeres},
Chen \cite{Chen16},
Chen, Wei, Zhang \cite{Chen_Wei_Zhang_08},
Chen and Chen \cite{Chen_Chen_10}, \cite{Chen_Chen_11},
Croydon and De Ambroggio \cite{Croydon_Triple},
Barlow, Peres and Sousi \cite{BarlowPeresSousi},
Gurel-Gurevich and Nachmias \cite{GG_M},
Hutchcroft and Peres \cite{HP15},
Watanabe \cite{W23}
and Roy, Takei and Tanemura \cite{Roy_Elephant_24}
for collisions of random walks on wedge combs, percolation clusters of $\Z^2$, some trees
including the three-dimensional uniform spanning tree, and
for elephant random walks.
See e.g. the introduction of
\cite{Croydon_Triple} for a more detailed recent review
of collisions of random walks on graphs.

Collisions have also been investigated recently for random walks in random environments on $\Z$.
In particular,
Gallesco \cite{G13} considered meetings of two Sinai walks, whereas
Gantert, Kochler and P\`ene \cite{NMFa} studied collisions only at the origin of $d$ independent
recurrent random walks in the same environment.
This was generalized by Devulder, Gantert and P\`ene \cite{Devulder_Gantert_Pene}
who provided a necessary and sufficient condition for several independent recurrent simple random walks
and recurrent random walks in random environments (Sinai walks) to meet simultaneously almost surely a finite (or infinite) number of times.
See also Kim and Kusuoka \cite{Kim_Kusuoka}
for related questions for Brownian motions and Brox diffusions
(we recall that the Brox diffusion is generally considered as a continuous time and space analogue of Sinai  walks),
and Halberstam and Hutchcroft \cite{HH22}
for collisions of random walks on some random conductance models.
Finally, Devulder, Gantert and P\`ene \cite{DGP_Collision_Transient} proved that $d$ independent random walks in the same environment in $\Z$
meet infinitely often, whatever $d$ is, in the transient subballistic case (when $0<\k<1$,  see below for a definition of $\k$).
The aim of the present paper is to extend this result when $0<\k<1/2$,
by adding a recurrent simple random walk.

We also mention that some applications of collisions of random walks in physics and in biology
are given in Campari and Cassi \cite{CampariCassi},
including some applications which can involve collisions of different kinds of random walks,
for example for prey-predator interactions, foraging, chemical reaction kinetics or pharmacokinetics.

\subsection{Model and results}
\label{Sub_Sec_12}
We consider a collection $\o:=(\o_x)_{x\in\Z}$  of i.i.d random variables,
taking values in $(0,1)$ and with joint law $\p$.
A realization of $\o$ is called an {\it environment}.
Let $d\geq 1$, and $(x_1,\dots, x_d, z_1)\in\Z^{d+1}$.
Conditionally on $\o$, we consider $(d+1)$ independent Markov chains
$\big(S_n^{(j)}\big)_{n\in\N}$, $1\leq j \leq d$ and $(Z_n)_{n\in\N}$,
defined  by $S_0^{(j)}=x_j$ for $1\leq j \leq d$,
$Z_0=z_1$
$\po^{(x_1,\dots, x_d, z_1)}$-almost surely
and for every $n\in\N$, $y\in\Z$ and $z\in\Z$,
\begin{equation}\label{probatransition}
    \po^{(x_1,\dots, x_d, z_1)}\big(S_{n+1}^{(j)}=z|S_n^{(j)}=y\big)
=
\left\{
\begin{array}{ll}
\o_y & \text{ if } z=y+1,\\
1-\o_y & \text{ if } z=y-1,\\
0      & \text{ otherwise},
\end{array}
\right.
\quad
    j\in\{1,\dots,d\},
\end{equation}
and
\begin{equation}
\label{probatransition2}
    \po^{(x_1,\dots, x_d, z_1)}\big(Z_{n+1}=z|Z_n=y\big)
=
\left\{
\begin{array}{ll}
1/2 & \text{ if } z=y+1 \text{ or } z=y-1,\\
0      & \text{ otherwise}.
\end{array}
\right.
\end{equation}
Each $\big(S_n^{(j)}\big)_{n\in\N}$ is called a {\it random walk in random environment} (RWRE),
whereas $Z:=(Z_n)_{n\in\N}$ is a (recurrent) simple random walk.
Thus, under $\po^{(x_1,\dots, x_d, z_1)}$,
$S^{(j)}:=\big(S_n^{(j)}\big)_{n\in\N}$, $1\leq j \leq d$, are $d$ independent random walks in the same (random) environment $\o$,
and are independent of the simple random walk $(Z_n)_{n\in\N}$.
Here and
in the sequel, $\N:=\{0,1,2,\dots\}$ denotes the set of nonnegative integers.

The probability $\po^{(x_1,\dots, x_d, z_1)}$ is called the {\it quenched law}.
We also define the {\it annealed law} as follows:
$$
    \P^{(x_1,\dots,x_d,z_1)}(\cdot):=\int\po^{(x_1,\dots,x_d,z_1)}(\cdot)\p(\textnormal{d}\o).
$$
Notice in particular that for $1\leq j \leq d$,  $\big(S_n^{(j)}\big)_{n\in\N}$ is not Markovian under $\P^{(x_1,\dots,x_d,z_1)}$.

We denote by $\E$ and $\eo^{(x_1,\dots,x_d, z_1)}$
the expectations under
$\p$ and $\po^{(x_1,\dots,x_d,z_1)}$
respectively. We also define
$$
    \rho_x
:=
    \frac{1-\o_x}{\o_x},
\qquad
    x\in\Z.
$$
We assume that there exists
an $\varepsilon_0\in(0,1/2)$ such that
\begin{equation}\label{UE}
    \p\big[\varepsilon_0\leq \omega_0 \leq 1-\varepsilon_0\big]
=
    1.
\end{equation}
Hypothesis \eqref{UE} is standard for RWRE
and is called the {\it ellipticity condition}.
Solomon \cite{S75} proved that each $S^{(j)}$ is almost surely recurrent
when $\E(\log \rho_0)=0$, and  transient to the right
when $\E(\log \rho_0)<0$.
Here and throughout the paper, $\log$ denotes the natural logarithm.
In this paper we only consider transient and subballistic RWRE.
More precisely, we assume that
\begin{equation}\label{Condition_Neg}
\E(\log \rho_0)<0
\end{equation}
and that
\begin{align}
& \label{ConditionA}
    \text{ there exists $\k>0$ such that }
    \E\big(\rho_0^\k\big)=1.
\end{align}
It is well known that such a $\kappa$ is unique.
Also, Hypotheses \eqref{Condition_Neg} and \eqref{ConditionA}
exclude the case when the $S^{(j)}$ are symmetric or non symmetric simple random walks (for which $\o_x$ does not depend on $x$).
Under our assumptions \eqref{UE}, \eqref{Condition_Neg} and \eqref{ConditionA}, since $\E(\log \rho_0)<0$,  each $\big(S_n^{(j)}\big)_{n\in\N}$ is a.s.
transient to the right. If moreover $0<\k<1$,
it has been proved by Kesten, Koslov and Spitzer  \cite{KestenKozlovSpitzer}
that, under the additional hypothesis that $\log\rho_0$ has a non-arithmetic distribution,
$S_n^{(j)}$ is asymptotically of order $n^{\k}$ as $n\to+\infty$ for each $1\leq j \leq d$
(we do not make this non-arithmetic assumption in the present paper).
We refer to Zeitouni \cite{Z01}, R{\'e}v{\'e}sz \cite{Revesz}
 and Shi \cite{S2}
for an overview of results proved for RWRE before 2005.
See e.g. \cite{Devulder_LLT_ArXiv} for more recent references.

\medskip

Our main result is the following.

\medskip

\begin{theorem}
\label{Theorem1Meeting_2}
Assume \eqref{UE}, \eqref{Condition_Neg} and \eqref{ConditionA}, with $0<\k<1/2$.
Let $x_1,\dots,x_d, z_1\in \mathbb Z$, with the same parity,
i.e.
$(x_1,\dots,x_d,z_1)\in[(2\Z)^{d+1}\cup (2\Z+1)^{d+1}]$,
where $d\geq 1$.
Then $\P^{(x_1,\dots,x_d,z_1)}$-almost surely, there exist infinitely many $n\in\N$ such that
$$
    S_n^{(1)}= S_n^{(2)}=\ldots=S_n^{(d)}=Z_n.
$$
\end{theorem}

To the best of our knowledge, this is the first example where there are infinitely many collisions
between a recurrent random walk, $(Z_n)_{n\in\N}$, and (arbitrary many) transient ones, the $\big(S_n^{(j)}\big)_{n\in\N}$.
Indeed, this result is valid for each value of $d\geq 1$.

In several other papers, see e.g. \cite{HP15} and the introduction of \cite{HH22},
one goal was to make, under appropriate hypotheses,
a link between the recurrence of a random walk
and the property that some independent copies of this random walk meet (i.e. collide)
almost surely infinitely often.
This follows from P{\'o}lya \cite{Polya}'s observation that
one simple random walk in $\Z^d$ is recurrent if and only if two independent simple random walks in $\Z^d$ meet infinitely often.
On the contrary, our Theorem \ref{Theorem1Meeting_2} gives an example of infinite collision property for a mixture of recurrent and transient random walks.

Our Theorem \ref{Theorem1Meeting_2} extends, only in the case $0<\kappa<1/2$,
the main result of Devulder et al.
\cite[Theorem 1.1]{DGP_Collision_Transient},
by adding the (recurrent) simple random walk $(Z_n)_{n\in\N}$.
It also continues the study started in \cite{Devulder_Gantert_Pene} of collisions between simple random walks and RWRE.


However, the infinite collision property stated in Theorem \ref{Theorem1Meeting_2}
does not hold when $\k>1/2$. Indeed, we prove:

\begin{proposition}
\label{Prop_kappa_grand}
Assume \eqref{UE}, \eqref{Condition_Neg} and \eqref{ConditionA}, with $\k>1/2$. Let $x_1,\dots,x_d, z_1\in \mathbb Z$,
where $d\geq 1$.
Then $\P^{(x_1,\dots,x_d,z_1)}$-almost surely, the number of $n\in\N$ such that
$$
    S_n^{(1)}= S_n^{(2)}=\ldots=S_n^{(d)}=Z_n
$$
is finite.
\end{proposition}

Heuristically,
the reason why
these random walks almost surely meet only a finite number of times
when $1/2<\kappa<1$
is that by \cite{KestenKozlovSpitzer},
with the additional assumption that $\log\rho_0$ has a non-arithmetic distribution,
each $S_n^{(j)}$ is asymptotically of order $n^{\k}$ as $n\to+\infty$ (see also Fribergh et al. \cite{FGP10} after their Theorem 1.1),
which is much larger than
$\sqrt{2 n \log\log n}$
and then much larger than $Z_n$
by the law of iterated logarithm for simple random walks.
Then we will prove that almost surely, $S_n^{(1)}=Z_n$ only a finite number of times.
This will also be true when $\kappa\geq 1$, since in this case $S_n^{(1)}$ is even larger.

On the contrary, when $0<\k<1/2$, for each $1\leq j\leq d$,
$S_n^{(j)}$ is negligible compared to $\sqrt{n}$ for large $n$,
which explains intuitively why $S^{(j)}$ can collide infinitely often with $Z$
even though one of them is transient while the other one is recurrent.
However, proving that the $d+1$ random walks in Theorem \ref{Theorem1Meeting_2} almost surely meet simultaneously infinitely often requires some precise estimates on the environment and on the random walks.
Our proof makes use of some intermediate results of \cite{DGP_Collision_Transient},
which rely on  quenched techniques for transient RWRE.
For more information about such techniques, we refer to Enriquez et al. \cite{ESZ2} and \cite{ESZ3},
Peterson et al. \cite{Peterson_Zeitouni}
and Dolgopyat et al. \cite{Dolgopyat_Goldsheid}.
See also Andreoletti et al. \cite{AndreolettiDevulder} and \cite{AndreolettiDevulderVechambre}
for quenched techniques for the continuous time and space analogue of transient RWRE with the point of view of $h$-extrema.

The recurrent case, that is, when $\E(\log \rho_0)=0$, has already been treated in \cite{Devulder_Gantert_Pene}.
Also, by symmetry, our Theorem \ref{Theorem1Meeting_2} and Proposition \ref{Prop_kappa_grand} remain valid when
our hypothesis $\E(\log \rho_0)<0$
is replaced by $\E(\log \rho_0)>0$, with the conditions $\k>0$, $0<\k<1/2$ and $\k>1/2$ replaced respectively by
$\k<0$, $-1/2<\k<0$ and $\k<-1/2$
in \eqref{ConditionA}, Theorem \ref{Theorem1Meeting_2} and Proposition \ref{Prop_kappa_grand}.

The proofs are provided in Section \ref{Sect_Proofs}.
Some extensions and open questions are given in Section \ref{Sect_Extensions_Questions}.

\section{Proofs}\label{Sect_Proofs}
Before starting the proofs, we recall that the {\it potential} $(V(x),\ x\in\Z)$ is a function of the environment $\o$, which is defined as follows:
\begin{equation*}
    V(x)
:=\left\{
    \begin{array}{lr}
    \sum_{k=1}^x \log\frac{1-\o_k}{\o_k}
    &
    \textnormal{if } x>0,\\
    0 & \textnormal{if }x=0,\\
    -\sum_{k=x+1}^0 \log\frac{1-\o_k}{\o_k} &  \textnormal{if } x<0.
    \end{array}
    \right.
\end{equation*}

\subsection{Proof of Theorem \ref{Theorem1Meeting_2}}
We suppose that \eqref{UE}, \eqref{Condition_Neg} and \eqref{ConditionA} hold, with $0<\k<1/2$.
Let $d\geq 1$.
We first assume that
$x_1\in(2\Z), \dots, x_d\in(2\Z)$ and $z_1\in(2\Z)$.

We now introduce some random variables defined in \cite{DGP_Collision_Transient};
their precise definitions are not needed here, we only need
to know that they exist and that they satisfy some properties, which are proved in \cite{DGP_Collision_Transient}
and that we recall.

As in
\cite[start of Section 3]{DGP_Collision_Transient},
we fix some
$\varepsilon\in(0,\frac{1-\kappa}{2\kappa})$,
which is possible since $0<\k<1/2$.
For every $x\in\R$, we denote by $\lfloor x\rfloor$ the integer part of $x$.
We define, for $i\geq 1$, as in
\cite[eq. (3.1) and (3.2)]{DGP_Collision_Transient},
\begin{equation}\label{eq_def_N_i}
    N_i
:=
    \big\lfloor C_0 i^{1+\varepsilon}(i!)^{\frac {1+\varepsilon}\kappa}\big\rfloor
\end{equation}
for some $C_0>1$ (the precise value of $C_0$ is not needed here; it is fixed  in
\cite[Fact 2.4]{DGP_Collision_Transient},
which extends Enriquez et al. \cite{ESZ2}, Lemma 4.9)
and
\begin{equation*}
    f_i
:=
    \log (N_i/C_0)-(1+\varepsilon)\log i
\sim_{i\to+\infty}
    \frac{1+\varepsilon}\kappa\log (i!)
\sim_{i\to+\infty}
    \frac{1+\varepsilon}\kappa i \log i\,  .
\end{equation*}
We also define by induction, as
(\cite{DGP_Collision_Transient} eq. (2.10))
and as in  \cite{ESZ2},
the weak descending ladder epochs for the potential $V$ as
\begin{equation}\label{eq_def_ei}
    e_0
:=
    0,
\qquad
    e_{i}
:=
    \inf\{k>e_{i-1},\ \ V(k)\leq V(e_{i-1})\},
\qquad
    i\geq 1.
\end{equation}
This allows us to introduce, as in
\cite[eq. (2.12)]{DGP_Collision_Transient}
and as in \cite{ESZ2},
the height $H_i$ of the excursion $[e_i,e_{i+1}]$ of the potential, which is
\begin{equation*}
    H_i
:=
    \max_{e_i\leq k\leq e_{i+1}}[V(k)-V(e_i)],
\qquad
    i\geq 0.
\end{equation*}
We now define by induction, as in
\cite[after eq. (3.2)]{DGP_Collision_Transient},
\begin{equation}\label{eq_def_sigma_i}
    \sigma(0)
:=
    -1,
\qquad
    \sigma(i)
:=
    \inf\{k > \sigma(i-1)\ :\ H_k\geq f_i\},
\qquad
    i\geq 1.
\end{equation}
and we recall that $\p$-almost surely, $\sigma(i)<\infty$ for all $i\in\N$.

We consider, for every integer $i\geq 1$, the random variable $b_i:=e_{\sigma(i)}$,
depending only on $\omega$ and $i$, as in
\cite[after eq. (3.2)]{DGP_Collision_Transient}.
By
\cite[Prop. 3.3, see also the definition of $\Omega_2^{(i)}$ before it] {DGP_Collision_Transient},
we have $\sum_{i\geq 1}\p\big[b_i> i e^{\kappa f_i}\big]<\infty$.
So by the Borel-Cantelli lemma,
$\p$- a.s.,
\begin{equation}\label{ineg_bi}
    b_i
\leq
    i e^{\kappa f_i}
\end{equation}
for large $i$.

We also consider, for $x\in\Z$, a single RWRE $(S_n)_{n\in\N}$ starting from $x$ in the environment $\o$,
and denote its quenched law  by $P_\omega^x$.
That is, $(S_n)_{n\in\N}$ satisfies \eqref{probatransition} with
$\po^{(x_1,\dots, x_d, z_1)}$, $S^{(j)}$ and $x_j$
replaced respectively by $P_\omega^x$, $S$ and $x$.
We denote by $E_\omega^x$ the expectation under $P_\omega^x$.
We also introduce
$$
    \tau(u)
:=
    \min\{k\in\N, \ S_k=u\}, \qquad u\in\Z.
$$
We recall the following results from \cite{DGP_Collision_Transient}.


\begin{lemma}\label{Lem_Omega_Tilde}
There exists a set of environments $\widetilde\Omega$,
which has probability $\p\big(\widetilde\Omega\big)=1$.
Also, there exists a (random) strictly increasing sequence $(i(n))_{n\in\N}$ of integers, defined on $\widetilde\Omega$ and depending
only on the environment $\omega$,
with $i(0)\geq 1$.
Moreover, there exists $C_8>0$ such that, for every $\omega\in\widetilde\Omega$,
\begin{equation}
\label{MINO1}
    \liminf_{n\rightarrow +\infty}\inf_{k\in [N_{i(n)}/2,2N_{i(n)}[\cap (2\mathbb Z)}
\ \
    P_\omega^{b_{i(n)}}\big[S_k=b_{i(n)}\big]
\ge
    C_8\, .
\end{equation}
Finally
for every $x\in\Z$, for $\p$-almost every environment $\omega$,
\begin{equation}
\label{eq_Cor_42}
    P_\omega^x[\tau(b_{i(n)})\le N_{i(n)}/10]
\to_{n\to+\infty
    }1.
\end{equation}
\end{lemma}

In the present paper, we just need some properties satisfied by $\widetilde\Omega$ and $(i(n))_{n\in\N}$,
gathered in the previous lemma,
but we do not need their precise definitions.

\begin{proof}[Proof of Lemma \ref{Lem_Omega_Tilde}  ]
First, $\widetilde\Omega$ and $(i(n))_{n\in\N}$ are defined in \cite[Proposition  3.5]{DGP_Collision_Transient},
which gives some of their properties. Also notice that in the proof of
\cite[Proposition  3.5]{DGP_Collision_Transient} with its notation,
we have in fact $i(0)=\inf{\ I}\geq i_0\geq 1$.
Also, \eqref{MINO1} and \eqref{eq_Cor_42}
are proved respectively in \cite[Lem. 4.7]{DGP_Collision_Transient}
and \cite[Corollary 4.2]{DGP_Collision_Transient}.
\end{proof}
In the previous lemma, the constant $C_8$ only depends on the law $\p$ of the environment
since in the proof of
\cite[Lem. 4.7]{DGP_Collision_Transient}
it is equal to $1/(8C_2)$,
with $C_2>0$ depending only on the law $\p$ (see \cite{DGP_Collision_Transient}, Lem. 2.5
and its proof).


Using the definition \eqref{eq_def_N_i} of $N_i$,
we now define, as in
\cite[before Lemma 4.8]{DGP_Collision_Transient},
\begin{equation}\label{eq_def_N_prime}
    N'_{i(n)}
:=
    N_{i(n)}-\mathbf 1_{\{(N_{i(n)}-b_{i(n)})\in (1+2\mathbb Z)\}},
\end{equation}
for every $n\in\N$ and every $\omega\in\widetilde\Omega$,
so that $N'_{i(n)}$ and $b_{i(n)}$ have the same parity.



Since $0<\k<1/2$, we can fix
$\delta\in(0,1/2-\k)$.
Also, notice that $\p$- a.s. for large $n$, first using \eqref{ineg_bi},
\begin{eqnarray}
    b_{i(n)}
& \leq &
    i(n) \exp[\k f_{i(n)}]
\nonumber\\
& = &
    i(n) \exp[\k(\log (N_{i(n)}/C_0)-(1+\varepsilon)\log [i(n)] )]
\nonumber\\
& = &
    (N_{i(n)}/C_0)^\k (i(n))^{1-\k(1+\varepsilon)}
\nonumber\\
& \leq &
    (N_{i(n)}/C_0)^{\k+\delta}
\leq
    (1/2)\sqrt{N_{i(n)}'}
\label{eq_finale}
\end{eqnarray}
because $N_i\geq (i!)^{(1+\e)/\k}$ for large $i$.
So, $\p$- a.s., $b_{i(n)}-z_1\leq \sqrt{N_{i(n)}'}$  for large $n$.
Also,
$e_0=0$ and $e_i> e_{i-1}$ for $i\geq 1$ by \eqref{eq_def_ei},
thus $e_i\geq i$ for every $i\in\N$ by induction.
Similarly, $\sigma(0)=-1$ and $\sigma(i)\geq \sigma(i-1)+1$ for $i\geq 1$ by \eqref{eq_def_sigma_i}
so $\sigma(i)\to_{i\to+\infty}+\infty$. Since $(i(n))_{n\in\N}$ is a strictly increasing sequence of integers, all this gives
$b_{i(n)}=e_{\sigma(i(n))}\to_{n\to+\infty}+\infty$.
As a consequence, we get, $\p$- a.s. for large $n$,
\begin{equation}\label{Ineg_b_i_n}
    0
\leq
    b_{i(n)}-z_1
\leq
    \sqrt{N_{i(n)}'}.
\end{equation}
We now introduce
\begin{eqnarray*}
    T_Z(u)
& := &
    \min\{k\in\N,\ Z_k=u\},\qquad u\in\Z,
\\
    T_Z(u,v)
& := &
    \min\{k\in\N,\ Z_{k+T_Z(u)}=v\},\qquad u\in\Z,\ v\in\Z,
\\
    U_n
& := &
    \min\{k\geq N_{i(n)},\ Z_k =b_{i(n)}\},
\qquad
    n\in\N.
\end{eqnarray*}
We have, $\p$- a.s. for large $n$,
\begin{eqnarray}
&&
    \po^{(x_1,\dots, x_d, z_1)}
    (U_n\leq 2N_{i(n)})
\nonumber\\
& = &
    \po^{(x_1,\dots, x_d, z_1)}
    (\exists
    N_{i(n)}
    \leq k \leq 2N_{i(n)}, \ Z_k =b_{i(n)})
\nonumber\\
& = &
    \po^{(x_1,\dots, x_d, 0)}
    (\exists N_{i(n)}
    \leq k \leq 2N_{i(n)}, \ Z_k =b_{i(n)}-z_1)
\nonumber\\
& \geq &
    \po^{(x_1,\dots, x_d, 0)}
    \Big(
        -\sqrt{N_{i(n)}}
        \leq Z_{4\lfloor N_{i(n)}/3\rfloor}
        \leq 0
        \leq b_{i(n)}-z_1
        \leq \sqrt{N_{i(n)}'}
        \leq Z_{2N_{i(n)}}
    \Big)
\nonumber\\
& \geq &
    \po^{(x_1,\dots, x_d, 0)}
    \Big(-\sqrt{N_{i(n)}}\leq Z_{4\lfloor N_{i(n)}/3\rfloor}\leq 0,
    \ Z_{2N_{i(n)}}\geq \sqrt{N_{i(n)}}
    \Big)
\nonumber\\
& \geq &
    C_1
\label{Ineg_Proba_Q_U_n}
\end{eqnarray}
for some constant $C_1>0$.
Indeed, we used, in the first inequality, the fact that $(Z_p)_{p\in\N}$ is a nearest neighbour random walk, so
$(Z_p)_{p\in[4\lfloor N_{i(n)}/3\rfloor, 2N_{i(n)}]\cap\N}$
and a fortiori
$(Z_p)_{p\in[N_{i(n)}, 2N_{i(n)}]}\cap\N$
take
all the values between $Z_{4\lfloor N_{i(n)}/3\rfloor}$ and $Z_{2N_{i(n)}}$
and in particular $b_{i(n)}-z_1$ when
$
Z_{4\lfloor N_{i(n)}/3\rfloor}
\leq
    b_{i(n)}-z_1
\leq
    Z_{2N_{i(n)}}
$.
Also, we used \eqref{Ineg_b_i_n} and \eqref{eq_def_N_prime}
so that $N_{i(n)}\geq N_{i(n)}'$
to prove the second inequality,
and
the last inequality follows
e.g. by
Donsker's theorem.

Now, extending
\cite[Lem. 4.8]{DGP_Collision_Transient},
we have the following result.

\begin{lemma}
\label{Lemma_Proba_etre_en_b_Elargi}
Let $x\in(2\Z)$.
For $\p$-almost every environment $\omega$,
$$
    \liminf_{n\rightarrow +\infty}
    \min_{k\in[3N_{i(n)}/4,\, 2N_{i(n)}]\cap(2\Z+b_{i(n)})}
    P_\omega^x\big[S_k=b_{i(n)}\big]
\ge
    \frac{C_8}2.
$$
\end{lemma}

\begin{proof}[Proof of Lemma \ref{Lemma_Proba_etre_en_b_Elargi}]
Define $\mathcal F_m:=\sigma(S_0, S_1, \dots, S_m, \omega)$ for $m\in\N$.
Let $x\in(2\Z)$ and $\omega\in\widetilde\Omega$.
In what follows,
for every $k\in\N$,
$P_\omega^{b_{i(n)}}\left[S_\ell=b_{i(n)}\right]_{|\ell=k}$
denotes
$P_\omega^{b_{i(n)}}\left[S_k=b_{i(n)}\right]$.
Due to the strong Markov property, for large $n$,
for every $k\in[3N_{i(n)}/4, 2N_{i(n)}]\cap(2\Z+b_{i(n)})$,
\begin{eqnarray}
    P_\omega^x\big[S_k=b_{i(n)}\big]
&=&
    E_\omega^x\big[ P_\omega^x\big(S_k=b_{i(n)}\,\big|\,\mathcal F_{\tau(b_{i(n)})}\big)\big]
\nonumber\\
&\ge &
    E_\omega^x\left[ \mathbf 1_{\{\tau(b_{i(n)})\le N_{i(n)}/10\}}P_\omega^{b_{i(n)}}\left[S_\ell=b_{i(n)}\right]_{|\ell=k-\tau(b_{i(n)})}\right]
\nonumber\\
&\ge&
    \frac{C_8}{2}\, P_\omega^x\left[\tau(b_{i(n)})\le \frac {N_{i(n)}}{10}\right]\,
\label{Ineg_Proba_Egal_bin}
\end{eqnarray}
    due to \eqref{MINO1},
since $\big(k-\tau(b_{i(n)})\big)\in[N_{i(n)}/2,2N_{i(n)}]\cap(2\Z)$
on the event $\{\tau(b_{i(n)})\le N_{i(n)}/10\}$.
Finally,
\eqref{eq_Cor_42} combined with \eqref{Ineg_Proba_Egal_bin} prove the lemma.
\end{proof}

As a consequence, $\p$- a.s. for large $n$, defining $\mathcal{G}_m:=\sigma(Z_0, Z_1,\dots, Z_m,\o)$ for $m\geq 0$,
\begin{eqnarray}
&&
    \po^{(x_1,\dots, x_d, z_1)}\big(\exists k\in[N_{i(n)}, 2N_{i(n)}],
    \ S_k^{(1)}=S_k^{(2)}=\dots=S_k^{(d)}=Z_k\big)
\nonumber\\
& \geq &
    \po^{(x_1,\dots, x_d, z_1)}\big(S_{U_n}^{(1)}=S_{U_n}^{(2)}=\dots=S_{U_n}^{(d)}=Z_{U_n}, \ U_n\leq 2N_{i(n)}\big)
\nonumber\\
& = &
    \eo^{(x_1,\dots, x_d, z_1)}\big[  \un_{\{U_n\leq 2N_{i(n)}\}}
        \po^{(x_1,\dots, x_d, z_1)}\big(S_{U_n}^{(1)}=S_{U_n}^{(2)}=\dots=S_{U_n}^{(d)}=b_{i(n)}\mid \mathcal{G}_{U_n}\big)
    \big]
\nonumber\\
& = &
    \eo^{(x_1,\dots, x_d, z_1)}\big[
        \un_{\{U_n\leq 2N_{i(n)}\}}
        \po^{(x_1,\dots, x_d, z_1)}\big(S_k^{(1)}=S_k^{(2)}=\dots=S_k^{(d)}=b_{i(n)}\big)_{| k=U_n}
    \big]
\nonumber\\
& \geq &
    \eo^{(x_1,\dots, x_d, z_1)}\bigg[
        \un_{\{U_n\leq 2N_{i(n)}\}}
        \min_{k\in[N_{i(n)}, 2N_{i(n)}]\cap(2\Z+b_{i(n)})}
        \prod_{j=1}^d
        \po^{x_j}\big(S_k=b_{i(n)}\big)
    \bigg]
\nonumber\\
& \geq &
    (C_8/4)^d
    \po^{(x_1,\dots, x_d, z_1)}(U_n\leq 2N_{i(n)})
\geq
    (C_8/4)^d
    C_1=:C_3>0,
\label{Ineg_def_C2}
\end{eqnarray}
where we used the definition of $U_n$ in the first inequality and in the first equality,
the independence of the $S^{(j)}$, $1\leq j\leq d$ with $Z$ and then with $U_n$ and $\mathcal{G}_{U_n}$
under $\po^{(x_1,\dots, x_d, z_1)}$
in the second equality,
the independence of $S^{(1)},\dots, S^{(d)}$, the definition of $U_n$ and $z_1\in(2\Z)$
so that $U_n\in(2\Z+b_{i(n)})$ under $\po^{(x_1,\dots, x_d, z_1)}$
in the second inequality,
and finally Lemma \ref{Lemma_Proba_etre_en_b_Elargi} and $x_j\in(2\Z)$, $1\leq j \leq d$,
followed by \eqref{Ineg_Proba_Q_U_n} in the last line.

For every $k\in\N$, we introduce the events:
$$
    A_k
:=
    \big\{S_k^{(1)}=S_k^{(2)}=\ldots = S_k^{(d)}=Z_k\big\},
\qquad
    E
:=
    \limsup\nolimits_n A_n
=
    \cap_{N\ge 0}\cup_{k\ge N}A_k.
$$
We also define, for $n\in\N$, the events
$$
    B_n
:=
    \cup_{k\in[N_{i(n)}, 2N_{i(n)}]} A_k,
\quad
    B_n'
:=
    \cup_{k\in[N_{i(n)}+1, 2N_{i(n)}+1]} A_k,
\quad
    B_n''
:=
    \cup_{k\in[N_{i(n)}, 2N_{i(n)}+1]} A_k.
$$
We now consider the case $(x_1,\dots,x_d,z_1)\in(2\Z+1)^{d+1}$.
Using the Markov property at time $1$ then applying \eqref{Ineg_def_C2},
we have  $\p$- a.s. for large $n$,
since $\big(S_1^{(1)},\dots, S_1^{(d)}, Z_1\big)$ belongs to a finite subset of $(2\Z)^{d+1}$
under $\po^{(x_1,\dots, x_d, z_1)}$,
\begin{equation*}
    \po^{(x_1,\dots, x_d, z_1)}(B_n')
=
    \eo^{(x_1,\dots, x_d, z_1)}\big[
        \po^{(S_1^{(1)},\dots, S_1^{(d)}, Z_1)}(B_n)
    \big]
\geq
    C_3.
\end{equation*}
This and \eqref{Ineg_def_C2} give, $\p$- a.s. for every
$(x_1,\dots,x_d,z_1)\in\mathcal{D}:=[(2\Z)^{d+1}\cup (2\Z+1)^{d+1}]$,
since $B_n\subset B_n''$ and $B_n'\subset B_n''$,
\begin{equation}\label{Ineg_def_C2_Impair}
    \liminf_{n\rightarrow +\infty}
    \po^{(x_1,x_2,\dots,x_d,z_1)}\big(B_n''\big)
\geq
    C_3>0.
\end{equation}
For $\p$-almost every $\omega$,
for every $(x_1,\dots,x_d, z_1)\in\mathcal{D}$,
 we have since $N_{i(n)}\to_{n\to+\infty}+\infty$,
\begin{eqnarray}
    \po^{(x_1,x_2,\dots,x_d,z_1)}\big(E\big)
&\ge&
    \po^{(x_1,x_2,\dots,x_d,z_1)}\left[\bigcap_{N\ge 0}\bigcup_{n\ge N} B_n''  \right]\nonumber\\
&=&
    \lim_{N\rightarrow +\infty} \po^{(x_1,x_2,\dots,x_d,z_1)}\left[\bigcup_{n\ge N}B_n''\right]\nonumber\\
&\ge&
    \liminf_{N\rightarrow +\infty}
    \po^{(x_1,x_2,\dots,x_d,z_1)}\big(B_N''\big)
\ge
    C_3>0
\label{MINO_2}
\end{eqnarray}
by \eqref{Ineg_def_C2_Impair}.

We conclude as in \cite{DGP_Collision_Transient}.
Observe that, for fixed $\omega$, $\big(S_n^{(1)},\dots,S_n^{(d)}, Z_n\big)_n$ is a Markov chain
and that $\mathbf 1_E$ is measurable, bounded
and stationary (i.e. shift-invariant).
Therefore,
applying a result of Doob (see e.g. \cite[Proposition V-2.4]{Neveu}),
\begin{equation}\label{Doob_2}
    \forall (x_1,\dots,x_d, z_1)\in\mathbb Z^{d+1},
\quad
   \lim_{n\rightarrow +\infty}P_\omega^{(S_n^{(1)},\dots,S_n^{(d)}, Z_n)}(E)
=
    \mathbf 1_E
\quad\
    P^{(x_1,\dots,x_d,z_1)}_\omega-\mbox{almost surely}.
\end{equation}
Moreover, for every $(x_1,\dots,x_d,z_1)\in\mathcal D$, $P_\omega^{(x_1,\dots,x_d,z_1)}$-almost surely,
for every $n\ge 0$,
we have $\big(S_n^{(1)},\dots,S_n^{(d)},Z_n\big)\in \mathcal D$.
So applying \eqref{MINO_2} to $\big(S_n^{(1)},\dots,S_n^{(d)},Z_n\big)\in \mathcal D$,
we obtain that, for $\p$-almost every $\omega$, (for which \eqref{MINO_2} holds),
for every $(x_1,\dots,x_d,z_1)\in\mathcal D$,
$P_\omega^{(x_1,\dots,x_d,z_1)}$-almost surely for every $n\geq 0$,
 $P_\omega^{(S_n^{(1)},\dots,S_n^{(d)}, Z_n)}(E)\geq C_3$.

Taking the limit as $n\to+\infty$ in the previous inequality and applying
\eqref{Doob_2}, we conclude that for $\p$-almost every environment $\omega$,
for every $(x_1,\dots,x_d,z_1)\in\mathcal D$,
$P_\omega^{(x_1,\dots,x_d,z_1)}$-almost surely,
$\mathbf 1_E\geq C_3>0$
and so $\mathbf 1_E=1$.
As a consequence,
for every $(x_1,\dots,x_d,z_1)\in\mathcal D$,
for $\p$-almost every $\omega$,
$P_\omega^{(x_1,\dots,x_d,z_1)}(E)=1$.
This concludes the proof of Theorem \ref{Theorem1Meeting_2}.
\hfill$\Box$

\subsection{Proof of Proposition \ref{Prop_kappa_grand}}
We suppose that \eqref{UE}, \eqref{Condition_Neg} and \eqref{ConditionA} hold. Let $d\geq 1$
and $(x_1,\dots,x_d, z_1)\in \mathbb Z^{d+1}$.

We first assume that $\k>1$.
In this case, $\E(\rho_0)<1$.
So, by Solomon \cite{S2}, there exists $v>0$ such that $S_n^{(1)}/n\to_{n\to +\infty} v$,
$\P^{(x_1,\dots,x_d,z_1)}$-almost surely.
Thus by the law of iterated logarithm for simple random walks (see e.g. R{\'e}v{\'e}sz \cite{Revesz} p. 31),
we have
$\P^{(x_1,\dots,x_d,z_1)}$-almost surely for large $n$,
$$
    S_n^{(1)} \geq (v/2) n > 2\sqrt{2 n \log\log n}> Z_n.
$$
Hence $\P^{(x_1,\dots,x_d,z_1)}$-almost surely for large $n$, $S_n^{(1)}\neq Z_n$, which proves
Proposition \ref{Prop_kappa_grand} when $\k>1$.

We now assume that $1/2<\k\leq 1$. To the best of our knowledge,
precise almost sure asymptotics of $S^{(1)}$, in terms of $\limsup$ or $\liminf$,  are not yet available
(as was noticed e.g. by Fribergh et al. \cite{FGP10} p. 45)
contrarily to
diffusions in a drifted Brownian potential (see Devulder \cite{Devulder_Max_Loc}, Corollary 1.10),
which are generally considered as the continuous time and space analogue of RWRE.

However, choosing $0<\e<\k-1/2$ and applying Fribergh et al.
\cite[eq. (1.7)]{FGP10},
for $\p$-almost every $\o$, we have
$\sum_{n=1}^\infty \po^x\big(S_n^{(1)}<n^{1/2+\e/2}\big)<\infty$
for $x=0$,
and so
$\sum_{n=1}^\infty \po^x\big(S_n^{(1)}<n^{1/2+\e}\big)<\infty$
for every $x\in\Z$.
So for $\p$-almost every $\o$,
we have $\po^{(x_1,\dots,x_d,z_1)}$-almost surely for large $n$,
$$
    S_n^{(1)}
\geq
    n^{1/2+\e}
>
    2\sqrt{2 n \log\log n}
>
    Z_n,
$$
once more by the law of iterated logarithm for the simple random walk $(Z_k)_k$.
Hence $\P^{(x_1,\dots,x_d,z_1)}$-almost surely for large $n$, $S_n^{(1)}\neq Z_n$, which proves
Proposition \ref{Prop_kappa_grand}
when $1/2<\k\leq 1$.
\hfill$\Box$

\section{Extensions and open questions}\label{Sect_Extensions_Questions}
We proved in Theorem \ref{Theorem1Meeting_2} and Proposition \ref{Prop_kappa_grand}
that under their hypotheses, there are
almost surely infinitely many
times $n$ for which
$
    S_n^{(1)}= S_n^{(2)}=\ldots=S_n^{(d)}=Z_n
$
when $0<\k<1/2$,
but only a finite number of such collisions when $\k>1/2$.
A natural question is: what happens when $\k=1/2$?
This might require more precise estimates than those of \cite{DGP_Collision_Transient}
(and e.g. \eqref{ineg_bi}).

Also, in the ballistic case (i.e. when $\k>1$), what can we say about the collisions between
independent random walks in the same environment and some simple random walks with the same speed
as the RWRE? (This question was in fact suggested by one referee for future work).
To this aim, it could be interesting to use a quenched local limit theorem (LLT) for ballistic
RWRE. Such a theorem has been proved by Dolgopyat et al. \cite{Dolgo_gold_13}
 when $\k>2$, for a slightly different model of RWRE (for which the probability
not to move is strictly positive at some locations, so that the RWRE is aperiodic).
However a quenched LLT is not yet available when $1<\k\leq 2$ to the best of my knowledge
(see e.g. \cite[Section 2]{Dolgo_gold_19} and the introduction of \cite{Devulder_LLT_ArXiv} for
recent reviews of LLT for RWRE).

We can also wonder what happens when there are several simple random walks,
as in \cite{Devulder_Gantert_Pene} with transient RWRE instead of recurrent ones.
That is, let $d\geq 1$, $p\geq 2$, $(x_1,\dots,x_d, z_1, \dots, z_p)\in \mathbb Z^{d+p}$,
and define $\omega$ as in Subsection \ref{Sub_Sec_12}.
Conditionally on $\o$, we consider $(d+p)$ independent Markov chains
$\big(S_n^{(j)}\big)_{n\in\N}$, $1\leq j \leq d$
and $\big(Z_n^{(i)}\big)_{n\in\N}$, $1\leq i \leq p$,
defined  by $S_0^{(j)}=x_j$ for $1\leq j \leq d$,
$Z^{(i)}_0=z_i$ for $1\leq i \leq p$,
and satisfying \eqref{probatransition}
and \eqref{probatransition2}
with $Z$ replaced by $Z^{(i)}$ for each $1\leq i \leq p$
and $\po^{(x_1,\dots, x_d, z_1)}$ replaced by
$\po^{(x_1,\dots, x_d, z_1,\dots, z_p)}$.
Then, we have the following result when there are $p\geq 3$ simple random walks.

\begin{proposition}
\label{Prop_kappa_grand_plusieurs_SW}
Assume \eqref{UE}, \eqref{Condition_Neg} and \eqref{ConditionA}, with $\k>0$.
Let $(x_1,\dots,x_d, z_1, \dots, z_p)\in \mathbb Z^{d+p}$,
where $d\geq 1$ and $p\geq 3$.
Then for almost every $\omega$,
the number of $n\in\N$ such that
$$
    S_n^{(1)}= S_n^{(2)}=\ldots=S_n^{(d)}=Z_n^{(1)}=\ldots=Z_n^{(p)}
$$
is $\po^{(x_1,\dots, x_d, z_1,\dots, z_p)}$-almost surely
finite.
\end{proposition}

\begin{proof}[Proof of Proposition \ref{Prop_kappa_grand_plusieurs_SW}]
The proof is very similar to the one of \cite[Thm. 1.2, case (i)]{Devulder_Gantert_Pene}, however we give it
here for the sake of completeness.
It is in fact valid for any
environment $\o\in(0,1)^\Z$.
We fix such an environment $\o$.
We have for large $n$,
\begin{eqnarray}
\label{Proba_Collisions_RWRE_SRW}
&&
    \po^{(x_1,\dots, x_d, z_1,\dots, z_p)}
    \big[
        S_n^{(1)}= S_n^{(2)}=\ldots=S_n^{(d)}=Z_n^{(1)}=\ldots=Z_n^{(p)}
    \big]
\\
& \leq &
    \po^{(x_1,\dots, x_d, z_1,\dots, z_p)}
    \big[
        S_n^{(1)}=Z_n^{(1)}=Z_n^{(2)}=Z_n^{(3)}
    \big]
\nonumber\\
& = &
    \sum_{k\in\Z}
    \po^{(x_1,\dots, x_d, z_1,\dots, z_p)}
    \big[
        S_n^{(1)}=Z_n^{(1)}=Z_n^{(2)}=Z_n^{(3)}=k
    \big]
\nonumber\\
& = &
    \sum_{k\in\Z}
    \po^{x_1}\big(S_n=k\big)
    \prod_{i=1}^3
    \po^{(x_1,\dots, x_d, z_1,\dots, z_p)}
    \big[
        Z_n^{(i)}=k
    \big]
\leq
    \sum_{k\in\Z}
    \po^{x_1}\big(S_n=k\big)
    \frac{C}{n^{3/2}}
=
    \frac{C}{n^{3/2}},
\nonumber
\end{eqnarray}
since for every $k\in\Z$ and $n\in\N$,
$
    \po^{(x_1,\dots, x_d, z_1,\dots, z_p)}
    \big[
        Z_{2n}^{(i)}=k
    \big]
=
    \po^{(0,\dots,0)}
    \big[
        Z_{2n}^{(i)}=k-z_i
    \big]
\leq
    \po^{(0,\dots,0)}
    \big[
        Z_{2n}^{(i)}=0
    \big]
\sim_{n\to+\infty}(\pi n)^{-1/2}
$,
e.g. by the local limit theorem for simple random walks or by Stirling's formula.
Hence the sum of the right hand sides of \eqref{Proba_Collisions_RWRE_SRW} is
finite, and Proposition \ref{Prop_kappa_grand_plusieurs_SW}
follows from the Borel-Cantelli lemma.
\end{proof}
However, the proof of Theorem \ref{Theorem1Meeting_2} cannot be extended to the case $p= 2$,
since we can prove,  using e.g. the local limit theorem,
Uchiyama \cite[Corollary 1.2 and Remarks 4 and 5 with our eq. \eqref{eq_finale}, or Theorem 1.4 and Remark 3]{Uchiyama},
and
Jain and Pruitt \cite[Section 4]{Jain_Pruitt}
or again
\cite[Corollary 1.2 and Remark 4 with $x=0$]{Uchiyama},
that
$
     P_\omega^{(x_1,\dots, x_d,z_1,z_2)}
     \big(
         \exists N_{i(n)}\leq k \leq 2N_{i(n)}, \ Z^{(1)}_k=Z^{(2)}_k=b_{i(n)}
     \big)
$
tends to $0$ a.s. as $n\to+\infty$.
Also, the proof of Proposition \ref{Prop_kappa_grand_plusieurs_SW} requires that $p\geq 3$.
This leads us to our last question:
is the number of collisions finite or infinite
when there are $d\geq 1$ (transient) RWRE and two (recurrent) simple random walks, i.e.
when $p=2$?



\begin{thebibliography}{99}


\bibitem
{AndreolettiDevulder}
    \textsc{Andreoletti, P. and Devulder, A.}:
    Localization and number of visited valleys for a transient diffusion in random environment.
    {\it Electron. J. Probab.} {\bf 20} (2015), no 56, 1--58.

\bibitem
{AndreolettiDevulderVechambre}
    \textsc{Andreoletti, P., Devulder, A. and V\'{e}chambre, G.}:
    Renewal structure and local time for diffusions in random environment.
    {\it ALEA Lat. Am. J. Probab. Math. Stat.} {\bf 13} (2016), 863--923.

\bibitem
{BarlowPeresSousi}
    \textsc{Barlow, M., Peres, Y. and Sousi, P.}:
    Collisions of random walks.
    {\it Ann. Inst. H. Poincar\'e Probab. Stat.} {\bf 48} (2012), no 4, 922--946.



\bibitem
{CampariCassi}
\textsc{Campari, R. and Cassi, D.}:
Random collisions on branched networks: How simultaneous diffusion prevents
encounters in inhomogeneous structures.
\textit{Physical Review E} \textbf{86} (2012), no. 2, 021110.




\bibitem
{Chen16}
    \textsc{Chen, X.}:
    Gaussian bounds and collisions of variable speed random walks on lattices with power law conductances.
    {\it Stoch. Process Appl.} \textbf{126} (2016), no 10, 3041--3064.


\bibitem
{Chen_Chen_10}
    \textsc{Chen, X. and Chen, D.}:
    Two random walks on the open cluster of $\Z^2$
    meet infinitely often.
    \textit{Science China Mathematics} \textbf{53} (2010), 1971--1978.

\bibitem
{Chen_Chen_11}
    \textsc{Chen, X. and Chen, D.}:
    Some sufficient conditions for infinite collisions of simple
    random walks on a wedge comb.
    {\it Electron. J. Probab.} \textbf{16} (2011), no 49, 1341--1355.

\bibitem
{Chen_Wei_Zhang_08}
    \textsc{Chen. D., Wei, B. and Zhang, F.}:
    A note on the finite collision property of random walks.
    \textit{Statist. Probab. Lett.}
    \textbf{78} (2008), 1742--1747.




\bibitem
{Croydon_Triple}
    \textsc{Croydon, D. A. and  De Ambroggio, U.}:
    Triple collisions on a comb graph.
    Preprint arXiv:2410.04882 (2024).


\bibitem
{Devulder_Max_Loc}
    \textsc{Devulder, A.}:
    The maximum of the local time of a diffusion process in a drifted {B}rownian potential.
    {\it S\'eminaire de Probabilit\'es} {\bf XLVIII} (2016), 123--177, Lecture Notes in Math., 2168, Springer.

\bibitem
{Devulder_LLT_ArXiv}
    \textsc{Devulder, A.}:
    Annealed local limit theorem for Sinai's random walk in random environment.
    Preprint arXiv:2309.13020 (2023).

\bibitem
{Devulder_Gantert_Pene}
    \textsc{Devulder, A., Gantert, N. and P\`ene, F.}:
    Collisions of several walkers in recurrent random environments.
    {\it Electron. J. Probab.} {\bf 23} (2018), no. 90, 1--34.

\bibitem
{DGP_Collision_Transient}
    \textsc{Devulder, A., Gantert N. and Pene F.}:
    Arbitrary many walkers meet infinitely often in a subballistic random environment.
    {\it Electron. J. Probab.} {\bf 24} (2019), no. 100, 1--25.


\bibitem
{Dolgopyat_Goldsheid}
   \textsc{Dolgopyat, D. and Goldsheid, I.}:
    Quenched limit theorems for nearest neighbour random walks in  1{D} random environment.
    {\it Comm. Math. Phys.} {\bf 315} (2012), 241--277.

\bibitem
{Dolgo_gold_13}
    \textsc{Dolgopyat, D. and Goldsheid, I.}:
    Local Limit Theorems for Random Walks in a 1D Random Environment.
    {\it Archiv der Mathematik} {\bf 101} (2013), 191--200.

\bibitem
{Dolgo_gold_19}
    \textsc{Dolgopyat, D. and Goldsheid, I.}:
    Local Limit Theorems for Random Walks in a Random Environment on a Strip.
    {\it Pure Appl. Funct. Anal.} {\bf 5} (2020), 1297--1318.


\bibitem
{Dvoretzky_Erdos}
    \textsc{Dvoretzky, A. and Erd{\"o}s, P.}:
    Some problems on random walk in space.
    Proceedings of the {S}econd {B}erkeley {S}ymposium on {M}athematical {S}tatistics and {P}robability, 1950 (1951),
    353--367,
    University of California Press, Berkeley and Los Angeles.



\bibitem
{ESZ2}
    \textsc{Enriquez, N., Sabot, C. and Zindy, O.}:
Limit laws for transient random walks in random environment on Z.
{\it Ann. Inst. Fourier (Grenoble)}  {\bf 59} (2009), 2469--2508.

\bibitem
{ESZ3}
    \textsc{Enriquez, N., Sabot, C. and Zindy, O.}:
    Aging and quenched localization for one dimensional random walks in random environment in the sub-ballistic regime.
    {\it Bull. Soc. Math. France}  {\bf 137} (2009), 423-452.



\bibitem
{FGP10}
    \textsc{Fribergh, A., Gantert, N. and Popov, S.}:
    On slowdown and speedup of transient random walks in random environment.
    {\it Probab. Theory Relat. Fields} {\bf 147} (2010), 43--88.

\bibitem
{G13}
    \textsc{Gallesco, C.}:
    Meeting time of independent random walks in random environment.
    {\it ESAIM Probab. Stat.} {\bf 17} (2013), 257--292.

\bibitem
{NMFa}
    \textsc{Gantert, N., Kochler, M. and P\`ene, F.}:
    On the recurrence of some random walks in random environment.
    \textit{ALEA Lat. Am. J. Probab. Math. Stat.} \textbf{11} (2014), 483--502.

\bibitem
{GG_M}
    \textsc{Gurel-Gurevich, O. and Nachmias, A.}:
     Nonconcentration of return times.
    {\it Ann. Probab.} {\bf 41} (2013), 848--870.


\bibitem
{HH22}
    \textsc{Halberstam, N. and Hutchcroft, T.}:
    Collisions of random walks in dynamic random environments.
    {\it Electron. J. Probab.} {\bf 27} (2022), no. 8, 1--18.

\bibitem
{HP15}
\textsc{Hutchcroft, T. and Peres, Y.}:
    Collisions of random walks in reversible random graphs.
    \textit{Electron. Commun. Probab.} \textbf{20} (2015), no. 63, 1--6.

\bibitem
{Jain_Pruitt}
    \textsc{Jain, N. C. and Pruitt, W. E.}:
    The range of random walk.
    Proceedings of the {S}ixth {B}erkeley {S}ymposium on
              {M}athematical {S}tatistics and {P}robability ({U}niv.
              {C}alifornia, {B}erkeley, {C}alif., 1970/1971), {V}ol. {III}:
              {P}robability theory,
    (1972),
    31--50,
 Univ. California Press, Berkeley, CA.


\bibitem
{KestenKozlovSpitzer}
    \textsc{Kesten, H., Kozlov, M.V. and Spitzer, F.}:
    A limit law for random walk in a random
    environment. {\it Compositio Math.} {\bf 30} (1975), 145--168.

\bibitem
{Kim_Kusuoka}
    \textsc{Kim, D. and Kusuoka, S.}:
    Recurrence of direct products of diffusion processes in random media having zero potentials.
    {\it Electron. J. Probab.}  {\bf 25} (2020), no. 139, 1--18.



\bibitem
{KrishnapurPeres}
    \textsc{Krishnapur, M. and Peres, Y.}:
   Recurrent graphs where two independent random walks collide finitely often.
    {\it Electron. Comm. Probab} {\bf 9} (2004), 72--81.

\bibitem
{Lawler_Limic}
    \textsc{Lawler, G. F. and Limic, V.}
    Random walk: a modern introduction.
    Cambridge Studies in Advanced Mathematics {\bf 123}, Cambridge University Press,
    Cambridge (2010). xii+364 pp.

\bibitem
{Neveu} \textsc{Neveu, J.}:
    Bases math\'ematiques du calcul des probabilit\'es.
    Deuxi\`eme \'edition, Masson et Cie \'editeurs, Paris (1970).



\bibitem
{Peterson_Zeitouni}
    \textsc{Peterson, J. and Zeitouni, O.}:
    Quenched limits for transient, zero speed one-dimensional random walk in random environment.
    {\it Ann. Probab.} {\bf 37} (2009), 143--188.



\bibitem
{Polya} \textsc{P{\'o}lya, G.}:
    \"{U}ber eine Aufgabe der Wahrscheinlichkeitsrechnung betreffend die Irrfahrt im Stra\ss ennetz.
    \textit{Math. Ann.} \textbf{84} (1921), 149--160.


\bibitem
{PolyaInc} \textsc{P{\'o}lya, G.}:
   Collected papers, Vol IV. Edited by Gian-Carlo Rota, M. C. Reynolds and R. M. Shortt.
   \textit{MIT Press}, Cambridge, Massachusetts (1984).


\bibitem
{Revesz}
    \textsc{R{\'e}v{\'e}sz, P.}:
    {\it Random walk in random and non-random environments}, second edition.
    World Scientific, Singapore (2005).




\bibitem
{Roy_Elephant_24}
    \textsc{Roy, R., Takei, M. and  Tanemura H.}:
    How often can two independent elephant random walks on $Z$ meet?
    {\it Proc. Japan Acad. Ser. A Math. Sci.} {\bf 100} (10), 57--59.


\bibitem
{S2}
     \textsc{Shi, Z.}:
     Sinai's walk via stochastic calculus. {\it Panoramas et Synth\`eses}
     {\bf 12} (2001), 53--74,
     Soci\'et\'e math\'ematique de France.


\bibitem
{S75}
    \textsc{Solomon, F.}:
    Random walks in a random environment.
    {\it Ann. Probab.} {\bf 3} (1975), 1--31.

\bibitem
{Uchiyama}
    \textsc{Uchiyama, K.}:
    The first hitting time of a single point for random walks,
    {\it Electron. J. Probab.} {\bf 16} (2011),
     no. 71, 1960--2000.

\bibitem
{W23}
    \textsc{Watanabe, S.}:
    Infinite collision property for the three-dimensional uniform spanning tree.
    {\it Int. J. Math. Ind.} {\bf 15} (2023), Paper No. 2350005.


\bibitem
{Z01}
    \textsc{Zeitouni, O.}:
    Lecture notes on random walks in random environment.
    {\it \'Ecole d'\'et\'e de probabilit\'es de
    Saint-Flour 2001}. Lecture Notes in
    Math. {\bf 1837} (2004),
    189--312. Springer, Berlin.








\end{thebibliography}


\begin{acks}
I am grateful to the two referees for their very careful reading of the paper
and for useful comments and questions.
\end{acks}


\end{document}